\documentclass[a4paper,12pt]{amsart}
\usepackage{amsfonts}
\usepackage{amsmath}
\usepackage{amssymb}
\usepackage[a4paper]{geometry}
\usepackage{mathrsfs}
\usepackage[colorlinks]{hyperref}
\renewcommand\eqref[1]{(\ref{#1})} 
\numberwithin{equation}{section}
\theoremstyle{plain}
\newtheorem{thm}{Theorem}[section]

\newtheorem{lem}[thm]{Lemma}

\theoremstyle{definition}
\newtheorem{defn}[thm]{Definition}
\newtheorem{rem}[thm]{Remark}

\begin{document}
\title[Geometric maximizers of Schatten norms]
{Geometric maximizers of Schatten norms of some convolution type integral operators}

\author[Michael Ruzhansky]{Michael Ruzhansky}
\address{
	Michael Ruzhansky:
	\endgraf
	Department of Mathematics
	\endgraf
	Imperial College London
	\endgraf
	180 Queen's Gate, London SW7 2AZ
	\endgraf
	United Kingdom
	\endgraf
	{\it E-mail address} {\rm m.ruzhansky@imperial.ac.uk}
}
\author[Durvudkhan Suragan]{Durvudkhan Suragan}
\address{
  Durvudkhan Suragan:
   \endgraf
  Institute of Mathematics and Mathematical Modeling
  \endgraf
  125 Pushkin str., 050010 Almaty, Kazakhstan
  \endgraf
and 
  \endgraf
  RUDN University 
  \endgraf
  6 Miklukho-Maklay St, Moscow, 117198, Russia
  \endgraf
  {\it E-mail address} {\rm suragan@list.ru}}

\thanks{The first author was supported in parts by the EPSRC
	grant EP/K039407/1 and by the Leverhulme Grants RPG-2014-02 and RPG-2017-151. The second author was supported by the Ministry of Education and Science of the Russian Federation (the Agreement number 02.a03.21.0008). 
	No new data was collected or generated during the course of this research.}


\begin{abstract}
In this paper we prove that the ball is a maximizer of the Schatten $p$-norm
of some convolution type integral operators with non-increasing kernels among all domains of a given measure in $\mathbb R^{d}$.
We also show that the equilateral triangle has the largest Schatten $p$-norm among all triangles of a given area. 
Some physical motivations for our results are also presented.
\end{abstract}
     \maketitle
\section{Introduction}
\label{intro}

Let $\Omega\subset \mathbb R^{d}$ be an open bounded set.
We consider the integral operator $\mathcal{K}_{\Omega}:L^{2}(\Omega)\rightarrow L^{2}(\Omega)$ defined by
\begin{equation}
\mathcal{K}_{\Omega}f(x):= \int_{\Omega}K(|x-y|)f(y)dy,\quad f\in L^{2}(\Omega),
\label{3}
\end{equation}
which we assume to be compact.
Throughout this paper we assume that the kernel $K(|x|)$ is
(a member of $L^{1}_{loc}(\mathbb R^{d})$) real, positive and non-increasing, i.e.
that the function $K:[0,\infty)\to \mathbb R$ satisfies
\begin{equation}
K(\rho)>0\quad  \textrm{ for any } \quad \rho\geq 0,
\label{n1}
\end{equation}
and
\begin{equation}\label{n1a}
K(\rho_{1})\geq K(\rho_{2})\quad  \textrm{ if } \quad \rho_{1}\leq \rho_{2}.
\end{equation}

Since $K$ is a real and symmetric function, $\mathcal{K}_{\Omega}$ is a self-adjoint operator.
Therefore, all of its eigenvalues and characteristic numbers are real.
We recall that the characteristic numbers are the inverses of the eigenvalues.
Usually,  examples of operators $\mathcal{K}_{\Omega}$ often appear as solutions to differential equations. For example, the Peierls integral operator (also discussed in \eqref{eq19}), namely the operator
$$
\mathcal{P}_{\Omega}f(x)=\int_{\Omega}\frac{1}{4\pi}\frac{e^{-|x-y|}}{|x-y|^{2}}f(y)dy,\quad f\in L^{2}(\Omega),\;
\Omega\subset \mathbb R^{3},
$$
appears as the inverse to the one-speed neutron transport equation in $\Omega$. In these way,
the eigenvalues of the differential operator 
correspond to characteristic numbers of operators
$\mathcal{K}_{\Omega}$.

The characteristic numbers of $\mathcal{K}_{\Omega}$ may be enumerated in ascending
order of their modulus,
\begin{equation}\label{EQ:ordering}
|\mu_{1}(\Omega)|\leq|\mu_{2}(\Omega)|\leq\ldots,
\end{equation}
where $\mu_{i}(\Omega)$ is repeated in this series according to its multiplicity. We denote the
corresponding eigenfunctions by $u_{1}, u_{2},...,$ so that for each characteristic number $\mu_{i}$ there is a unique corresponding (normalised) eigenfunction $u_{i}$,
$$u_{i}=\mu_{i}(\Omega)\mathcal{K}_{\Omega}u_{i},\,\,\,\, i=1,2,\ldots.$$

By using the Feynman-Kac formula and spherical rearrangement Luttinger \cite{Lu} proved that the ball $B$ is a maximizer
of the partition function of the Dirichlet Laplacian among all domains of the same volume as $B$ for all positive values of time, i.e.
$$
\sum_{i=1}^{\infty}\exp(-t\mu_{i}^{D}(\Omega))\leq \sum_{i=1}^{\infty}
\exp(-t\mu_{i}^{D}(B)),\,\,\forall t>0 ,\,\,|\Omega|=|B|,
$$
where $\mu_{i}^{D}, i=1,2,\ldots,$ are the eigenvalues of the Dirichlet Laplacian.
From this, by using the Mellin transform one obtains
\begin{equation}\sum_{i=1}^{\infty}\frac{1}{[\mu_{i}^{D}(\Omega)]^{p}}\leq \sum_{i=1}^{\infty}\frac{1}{[\mu_{i}^{D}(B)]^{p}},\,\,
|\Omega|=|B|, \label{eq5}\end{equation}
when $p>d/2$, $\Omega\subset \mathbb R^{d}$.
See also Luttinger \cite{Lu1}.
We prove an analogue of this Luttinger's inequality for the integral operator $\mathcal{K}_{\Omega}$.

In addition to this analysis we also look at the isoperimetric inequalities in the case when
$\Omega$ ranges over polygons, more specifically, over the triangles. In the Dirichlet Laplacian case such problems have been treated by P{\'o}lya \cite{Po}, see also the exposition by P{\'o}lya and Szeg{\"o} \cite{Po1}.
Summarising our main results for operators $\mathcal{K}_{\Omega}$ with an additional condition on the kernel in both cases of arbitrary or polygonal domains, we prove the following facts:

\begin{itemize}
\item the $p$-Schatten norm of (positive) $\mathcal{K}_{\Omega}$ is maximized on the ball among all domains
of a given measure;
\item  the $p$-Schatten norm of (positive) $\mathcal{K}_{\Omega}$ is maximized on the equilateral triangle
among all triangles of a given measure.
\end{itemize}

For the logarithmic potential operator, i.e. for the operator
\begin{equation*}
\mathcal{L}_{\Omega}f(x)= \int_{\Omega}\frac{1}{2\pi}\ln\frac{1}{|x-y|}f(y)dy,\quad f\in L^{2}(\Omega),
\end{equation*}
for $\Omega\subset\mathbb R^{2}$, some results analogous to those of the present
paper have been obtained in \cite{Ruzhansky-Suragan:JMAA}, for the Riesz potential operator in \cite{Ruzhansky-Suragan:Newton} as well as for more general convolution type operators \cite{Ruzhansky-Suragan:UMN}. However, all previous papers on Schatten $p$-norms were only for integer values of $p$, therefore,   
the setting of this paper is different, that is, in the present paper in order to obtain geometric maximizer results for non-integer $p$ we define a special class of convolution type operators allowing for deeper analysis based on specific properties of
the kernel. We also note the related works \cite{Zoal1} and \cite{Zoal2} on isoperimetric inequalities for the logarithmic and more general potential operators.

In Section \ref{SEC:result} we present the main result of this paper.
Its proof will be given in Section \ref{SEC:proof}.
In Section \ref{SEC:6} we discuss briefly isoperimetric inequalities for polygons.

\medskip

The authors would like to thank Grigori Rozenblum for useful discussions on an early version of this paper.

\section{Main result and examples}
\label{SEC:result}

Let $H$ be a separable Hilbert space.
We denote the class of compact operators $P:H\rightarrow H$ by $\mathcal{S}^{\infty}(H)$.
Recall that the singular values $\{s_{n}\}$ of $P\in \mathcal{S}^{\infty}(H)$
are the eigenvalues of the positive operator $(P^{*}P)^{1/2}$ (see e.g.
Gohberg and Krein \cite{GK}).
The Schatten $p$-classes are defined as
$$\mathcal{S}^{p}(H):=\{P \in \mathcal{S}^{\infty}(H): \{s_{n}\}\in \ell^{p}\},\quad 1\leq p<\infty.$$
In $\mathcal{S}^{p}(H)$ the Schatten $p$-norm of the operator $P$ is defined as
\begin{equation}
\| P\|_{p}:=\left(\sum_{n=1}^{\infty}s_{n}^{p}\right)^{\frac{1}{p}}, \quad 1\leq p <\infty.\label{1}
\end{equation}
For $p=\infty$, we can set
$$
\| P\|_{\infty}:=\| P\|
$$
to be the operator norm of $P$ on $H$.
It is well known that for the positive self-adjoint operators the singular values are equal to the eigenvalues, and the corresponding eigenfunctions form a complete orthogonal basis on $H$.
As outlined in the introduction, we assume that $\Omega\subset \mathbb R^{d}$
is an open bounded set, and we consider compact integral operators
on $L^{2}(\Omega)$ of the form
\begin{equation}
\mathcal{K}_{\Omega}f(x)=\int_{\Omega}K(|x-y|)f(y)dy,\quad f\in L^{2}(\Omega),
\label{n16}
\end{equation}
where the kernel $K$ is real, positive and non-increasing, that is, $K$
satisfies \eqref{n1} and \eqref{n1a}. This implies that it
does not change its formula under the symmetric-decreasing rearrangement, see e.g.
Lieb and Loss \cite{LL}.

By $|\Omega|$ we will denote the Lebesque measure of $\Omega$.

The following analogue of the Rayleigh-Faber-Krahn inequality for the integral operator $\mathcal{K}_\Omega$ was given in \cite{Ruzhansky-Suragan:UMN} (see also \cite{Ruzhansky-Suragan:Newton} and \cite{Ruzhansky-Suragan:BMS}).

\begin{thm} \label{THM:1}
The ball $B$ is a minimizer of the first characteristic number (by modulus) of the
convolution type compact operator $\mathcal{K}_\Omega$ among all domains of
a given measure, i.e.
\begin{equation}\label{EQ:mu1}
\mu_{1}(B)\leq \mu_{1}(\Omega)
\end{equation}
for an arbitrary bounded open domain $\Omega\subset \mathbb R^{d}$ with $|\Omega|=|B|.$
\end{thm}

\begin{rem}\label{REM:op}
In other words Theorem \ref{THM:1} says that the operator norm
$\|\mathcal{K}_\Omega\|$ is maximized in a ball among all Euclidean domains of a given volume. We also note that since $\mu_{1}$ is the
characteristic number with the smallest modulus according to the ordering \eqref{EQ:ordering}, actually $\mu_{1}$ is positive.
Since the integral kernel of $\mathcal{K}$ is positive, the statement, sometimes called Jentsch's theorem, applies, see, e.g.,  \cite{RSV4}: the characteristic number $\mu_{1}$ of the convolution type compact operator $\mathcal{K}$ with the smallest modulus is positive and simple; the corresponding eigenfunction $u_{1}$ can be chosen positive. Therefore,  \eqref{EQ:mu1} is the inequality between positive numbers.
\end{rem}

\medskip

We now define a $\zeta$-kernel of the operator $\mathcal{K}_{\Omega}$: we say that
$\mathcal{K}_{\Omega}$ admits a $\zeta$-kernel if for all $\zeta\geq 0$ there exists
a real, positive and non-increasing function $K_{\zeta}(|x-y|)$, a member of $L^{1}_{loc}(\mathbb R^{d})$, such that
\begin{equation}\label{EQ:zk}
u_{i}(x)=(\mu_{i}+\zeta)\int_{\Omega}
K_{\zeta}(|x-y|)u_{i}(y)dy,\quad i=1,2,\ldots
\end{equation}
Here $\mu_{i}$ and $u_{i}$ are characteristic numbers and corresponding eigenfunctions of the operator $\mathcal{K}_{\Omega}$. Note that a real, positive and non-increasing function is called a symmetric-decreasing function.

\begin{thm} \label{THM:main}
Let $B$ be a ball centered at the origin and let $\Omega$
be a bounded open domain with $|\Omega|=|B|.$
Assume that the operators $\mathcal{K}_{\Omega}$ and $\mathcal{K}_{B}$ are positive,
admit $\zeta$-kernels, and $\mathcal{K}_{B}\in\mathcal{S}^{q}(L^{2}(B))$
for some $q>1$. Let $p_{0}$ be the smallest integer
$\geq q.$
Then
\begin{equation}
\|\mathcal{K}_{\Omega}\|_{p}\leq  \|\mathcal{K}_{B}\|_{p}
\label{3-2}
\end{equation}
for all $p_{0}\leq p< \infty$.
\end{thm}
Conditions for funding $q$ for which an integral operator belongs to $\mathcal{S}^{q}$ in terms of the regularity of its kernel were given, for example, in \cite{DR-JFA}, see also \cite{DR-CRAS}.

\subsection{An example}
If the inverse to the operator $\mathcal{K}_{\Omega}$ is an elliptic operator, then its $\zeta$-kernel may be found as the Laplace transform of the corresponding heat kernel (see \cite{dav}).
The $\zeta$-kernel assumption, of course, requires that the conditions for the positivity of the heat kernel are satisfied, see, e.g., Reed-Simon \cite{RSV4}, which, in particular, restricts this kind of examples to elliptic operators of order not higher than 2. However, many problems in mathematical physics are reducible to convolution type positive integral operators. Here we consider the one-speed neutron transport equation to illustrate an application of our results in mathematical physics.

We denote by $\psi(x, s)$ the flux of particles at the point $x$ in a steady state, moving in the direction $s = (s_{1},s_{2},s_{3})$, $|s|= 1$, with no sources. Then the function $\psi(x, s)$ satisfies the following integral-differential equation:

\begin{equation}\langle s,\nabla \psi\rangle +b \psi=\frac{\alpha h}{4\pi}\int_{S}\psi(x,\rho)d\rho,
\end{equation}
where $S$ is a surface of the unit sphere centered at 0 in $\mathbb R^{3}$, $\langle s,\nabla \psi\rangle=\sum_{i=1}^{3}s_{i}\frac{\partial \psi}{\partial x_{i}},$ $b= 1/l$, $l$ is the mean free path at the point $x$ and we assume that it is a constant $l>0$.
This is the one-speed neutron transport equation for stationary processes with isotropic scattering without sources.

For a complete description of the process of particle transport it is necessary to prescribe the behaviour of the flux of particles $\psi(x, s)$ on the boundary of this medium (boundary conditions).

Here, for convenience,  we assume that the domain $\Omega$ with boundary $\partial\Omega$, where the transport process is taking place, is convex. In this case a boundary condition of the
form
\begin{equation}\label{a2}
\psi(x, s)=0, \langle s,n_{x}\rangle <0,
\end{equation}
expresses the absence of a flux of particles incident on the domain $\Omega$ from the outside.
In the boundary condition \eqref{a2}, $n_{x}$ is a unit outer normal at $x$ and $\langle s,n_{x}\rangle$ is a scalar product.

Thus, we obtain the spectral problem

\begin{equation}\langle s,\nabla \psi\rangle +b \psi=\frac{\mu}{4\pi}\int_{S}\psi(x,\rho)d\rho
, \quad x\in\Omega\subset \mathbb R^{3}, \label{eq19}\end{equation}

\begin{equation}
\psi(x, s)=0, \langle s,n_{x}\rangle <0,x\in\partial\Omega,\label{eq20}
\end{equation}
where $\mu=\alpha h$.
It is known that the spectral problem \eqref{eq19}-\eqref{eq20} is equivalent to the Peierls integral equation (\cite{Vla}, for more details see also \cite{Vla1})

\begin{equation}u(x)=
\mu\int_{\Omega}\mathcal{P}(|x-y|)u(y)dy,\end{equation}
where $u(x)=\frac{1}{4\pi}\int_{S}\psi(x,\rho)d\rho$
and $\mathcal{P}(|x-y|)=\frac{1}{4\pi}\frac{e^{-b|x-y|}}{|x-y|^{2}}$ is the Peierls kernel.

Theorem \ref{THM:1} is valid for this operator. Moreover, one concludes that Theorem \ref{THM:1} can be interpreted in theory of one speed neutrons that a ball is a maximizer of the probability that the particles may be captured by the nucleus among all domains of a given volume in $\mathbb R^{3}$ (see \cite{Suragan}). In addition, since we have the following $\zeta$-kernel of the Peierls integral operator:
\begin{equation}\mathcal{P}_{\zeta}(|x-y|)=\frac{1}{4\pi}
\frac{e^{-(\zeta+b)|x-y|}}{|x-y|^{2}},\end{equation}
Theorem \ref{THM:main} applies in this case.

\section{Proof of Theorem \ref{THM:main}}
\label{SEC:proof}

For integer values of $p$ we obtain:
\begin{lem} \label{lem:2} If the Schatten $q$-norm of $\mathcal{K}_{B}$ is bounded in a ball $B\subset \mathbb R^{d},$
then
\begin{equation}
\sum_{j=1}^{\infty}\frac{1}{\mu_{j}^{p}(\Omega)}\leq
\sum_{j=1}^{\infty}\frac{1}{\mu_{j}^{p}(B)}\label{n7}
\end{equation}
for any integer $p\geq q$ and an arbitrary bounded open domain $\Omega$ with $|\Omega|=|B|.$
\end{lem}

\begin{proof}[Proof of Lemma \ref{lem:2}]
By using bilinear expansion of the iterated kernels (see, for example, \cite{Vla}) we obtain
\begin{equation}
\sum_{j=1}^{\infty}\frac{1}{\mu_{j}^{p}(\Omega)}=\int_{\Omega}...
\int_{\Omega}K(|y_{1}-y_{2}|)...K(|y_{p}-y_{1}|)dy_{1}...dy_{p},\,\,\, p\geq q, \,\,
p\in \mathbb N.\label{eq15}
\end{equation}
Indeed this is the trace of the $p$-th power of the operator $\mathcal{K}_\Omega.$
It follows from the Brascamp-Lieb-Luttinger inequality \cite{BLL} (see also \cite{LF}) that
$$
\int_{\Omega}...
\int_{\Omega}K(|y_{1}-y_{2}|)...K(|y_{p}-y_{1}|)dy_{1}...dy_{p}\leq$$
\begin{equation}\int_{B}...
\int_{B}K(|y_{1}-y_{2}|)...K(|y_{p}-y_{1}|)dy_{1}...dy_{p},\label{eq16}
\end{equation}
which proves
\begin{equation}
\sum_{j=1}^{\infty}\frac{1}{\mu_{j}^{p}(\Omega)}\leq
\sum_{j=1}^{\infty}\frac{1}{\mu_{j}^{p}(B)},\,\,\, p\geq q,\,\,\, p\in \mathbb N,\label{eq17}
\end{equation}
for any bounded open domain $\Omega\subset \mathbb R^{d}$ with $|\Omega|=|B|$.
Here we have used that the kernel $K$ is a symmetric-decreasing function.
\end{proof}

Since by the assumption $\mathcal{K}_\Omega$ is a positive operator, we have
\begin{equation}
\mu_{i}=|\mu_{i}|,\quad i=1,2,\ldots.
\end{equation}
As a consequence of Lemma \ref{lem:2} we obtain:
\begin{lem} \label{lem:3}
If the Schatten $q$-norm of $\mathcal{K}_{B}$ is bounded in a ball $B\subset \mathbb R^{d},$
then we have
\begin{equation}\label{n8} \sum_{i=1}^{\infty}\frac{2|\mu_{i}(\Omega)|}{|\mu_{i}(\Omega)|^{p-1}
(|\mu_{i}(\Omega)|^{2}-\zeta^{2})}\leq
\sum_{i=1}^{\infty}\frac{2|\mu_{i}(B)|}{|\mu_{i}
(B)|^{p-1}(|\mu_{i}(B)|^{2}-\zeta^{2})}
,\,\,|\zeta|<\mu_{1}(B),
\end{equation}
for any integer $p\geq q.$
\end{lem}

\begin{proof} [Proof of Lemma \ref{lem:3}]
By Theorem \ref{THM:1} we have that
$\mu_{1}(B)$ is the smallest of all $|\mu_{i}(B)|$ and $|\mu_{i}(\Omega)|,$ $|\Omega|=|B|,$ $i=1,2,....$
Therefore, for every $\zeta\in\mathbb R$ such that $|\zeta|<\mu_{1}(B)$ by using the geometric series we have

$$\sum_{i=1}^{\infty}\frac{|\mu_{i}(\Omega)|}{|\mu_{i}(\Omega)|^{p-1}(|\mu_{i}(\Omega)|^{2}-\zeta^{2})}=
\sum_{i=1}^{\infty}\frac{1}{|\mu_{i}(\Omega)|^{p}(1-\frac{\zeta^{2}}{|\mu_{i}(\Omega)|^{2}})}=$$
$$\sum_{i=1}^{\infty}\frac{1}{|\mu_{i}(\Omega)|^{p}}\sum_{n=0}^{\infty}
\left(\frac{\zeta^{2}}{|\mu_{i}(\Omega)|^{2}}\right)^{n}=
\sum_{n=0}^{\infty}\zeta^{2n}\sum_{i=1}^{\infty}\frac{1}{|\mu_{i}(\Omega)|^{p+2n}},\quad p\geq q.$$
Here the condition integer $p\geq q$ ensures convergence of the sum.
Now the claim follows from Lemma \ref{lem:2}, that is,
$$\sum_{n=0}^{\infty}\zeta^{2n}\sum_{i=1}^{\infty}\frac{1}{|\mu_{i}(\Omega)|^{p+2n}}\leq
\sum_{n=0}^{\infty}\zeta^{2n}\sum_{i=1}^{\infty}\frac{1}{|\mu_{i}(B)|^{p+2n}},$$
completing the proof.
\end{proof}

As a consequence of Lemma \ref{lem:3} we have the following inequality:
\begin{lem} \label{lem:4}
Under conditions of Lemma \ref{lem:3}, we have
\begin{equation}\label{n9}
\sum_{i=1}^{\infty}\frac{e^{-|\mu_{i}(\Omega)|t}}{|\mu_{i}(\Omega)|^{p-1}}\leq
\sum_{i=1}^{\infty}\frac{e^{-|\mu_{i}(B)|t}}{|\mu_{i}(B)|^{p-1}},\quad \forall t>0,
\end{equation}
for any bounded open domain $\Omega$ with $|\Omega|=|B|$ and any integer $p\geq q.$
\end{lem}

\begin{proof} [Proof of Lemma \ref{lem:4}]
Let $\mathcal{B}$ be the bilateral Laplace transform
\begin{equation}\label{EQ:Bil}
\mathcal{B}\{f(t)\}(\zeta)=\int_{-\infty}^{\infty}e^{-\zeta t}f(t)dt.
\end{equation}
Using the inverse bilateral Laplace transform (frequency shift of
unit step) we have
$$
\mathcal{B}^{-1}\left\{\frac{2 |\mu_{i}|}{|\mu_{i}|^{2}-\zeta^{2}}\right\}(t)=
e^{-|\mu_{i}| |t|},\,\,|\zeta|< \mu_{1}(B).
$$
Note that since $\mu_{1}(B)$ is the least eigenvalue of all $\mu_{i}(B)$ and $\mu_{i}(\Omega),$ $i=1,2,...,$ (see Theorem \ref{THM:1}) the above equality holds for all $\zeta$ such that $|\zeta|< \mu_{1}(B)$.
By applying $\mathcal{B}^{-1}$ to both sides of \eqref{n8}
we obtain \eqref{n9} (see Lemma \ref{pos:f} for $\mathcal{B}^{-1}$ preserving inequality \eqref{n8}).
\end{proof}

One might have a question concerning the proof of Lemma \ref{lem:4}, that is,
why does the inverse bilateral Laplace transform preserve the inequality
\eqref{n8}?
In other words, why is the inverse bilateral Laplace transform of a positive function positive?
Of course, this is not true in general.
However, for the Laplace transform
$$\mathcal{L}\{f(t)\}(\zeta)=\int_{0}^{\infty}e^{-\zeta t}f(t)dt,$$
the inverse Laplace transform of a positive function is positive for some
classes of functions,
that is, the following theorem is valid (see \cite[Theorem 2.3]{Br}).
\begin{thm}[\cite{Br}]\label{bryan}
Let $f$ be a continuous function on the interval $[0,\infty)$ which is of exponential order, that is, for some $b\in\mathbb R$
it satisfies
$$\sup_{t>0}\frac{|f(t)|}{e^{bt}}<\infty,$$
and let $F_{1}:=\mathcal{L}f$.
Then $f$ is non-negative if and only if
\begin{equation}\label{EQ:Lappos}
(-1)^{k} F^{(k)}_{1}(s)\geq 0 \quad \textrm{ for all } k\geq 0 \textrm{ and all } s>b.
\end{equation}
\end{thm}
In fact this positivity result follows directly from Post's inversion formula (see \cite{Post:1930})
\begin{equation}\label{Post}
f(t)=\lim_{k\rightarrow\infty}
\frac{(-1)^{k}}{k!}\left(\frac{k}{t}\right)^{k+1}F^{(k)}_{1}\left(\frac{k}{t}\right)
\end{equation}
for $t>0$.
If \eqref{EQ:Lappos} is valid then the expression on the right hand side of \eqref{Post} is non-negative.
Therefore, the limit $f(t)$ is necessarily non-negative for all $t$.

In the proof of Lemma \ref{lem:4}, we have the bilateral Laplace transform
$\mathcal{B}$ in \eqref{EQ:Bil},
which can be also reduced to the
Laplace transform
$\mathcal{L}.$
Therefore, it can be checked that in our special case the inequality is actually preserved
(for this particular function). To show it first we prove

\begin{lem} \label{+s} Let $p$ be an integer $\geq q.$ Then
\begin{equation}\label{fff4}
\sum_{j=1}^{\infty}\frac{1}{|\mu_{j}(\Omega)|^{p-1}(|\mu_{j}(\Omega)|
+\zeta)^{n}}\leq\sum_{j=1}^{\infty}\frac{1}{|\mu_{j}(B)|^{p-1}
(|\mu_{j}(B)|+\zeta)^{n}},
\end{equation}
for any $n\in \mathbb N$ and $0<|\mu_{j}(B)|+\zeta$, that is, $-|\mu_{1}(B)|<\zeta$.
\end{lem}

\begin{proof} [Proof of Lemma \ref{+s}]
Since by the assumption $\mathcal{K}_\Omega$ is a positive operator, we have
 $\mu_{i}=|\mu_{i}|$, $i=1,2,\ldots$, and let us rewrite \eqref{n7} in the form
\begin{equation}
\sum_{j=1}^{\infty}\frac{1}{|\mu_{j}(\Omega)|^{p}}\leq
\sum_{j=1}^{\infty}\frac{1}{|\mu_{j}(B)|^{p}},\,\, p\geq q,\,\, p\in \mathbb N.\label{ff3}
\end{equation}
Since we have added the same number $\zeta$ to denominators of both sides of \eqref{ff3}
by preserving the sign of each term we get the following analytic function in $(-\lambda_{1},\infty)$, with $\lambda_1=|\mu_{j}(B)|$:
\begin{equation}\label{ff4}
F_{n}(\zeta):=\sum_{j=1}^{\infty}
\frac{1}{|\mu_{j}(B)|^{p-1}(|\mu_{j}(B)|+\zeta)^{n}}
-\sum_{j=1}^{\infty}\frac{1}{|\mu_{j}(\Omega)|^{p-1}(|\mu_{j}(\Omega)|
+\zeta)^{n}}.
\end{equation}
First let us  show that the positivity of \eqref{ff4} follows from \eqref{ff3} when
$\zeta\in(-\lambda_{1},0].$
Notice that
\begin{equation}\label{rec}
F^{\prime}_{n}(\zeta)=-nF_{n+1}(\zeta),\,n=1,2,....
\end{equation}
Therefore, we obtain
\begin{equation}\label{Mac}
F_{n}(\zeta)=F_{n}(0)-nF_{n+1}(0)\zeta+
\frac{n(n+1)}{2!}F_{n+2}(0)\zeta^{2}-
\frac{n(n+1)(n+2)}{3!}F_{n+3}(0)\zeta^{3}+....
\end{equation}
By \eqref{ff3} we have $F_{n}(0)\geq 0$ for all positive integers $n$, so it is clear from \eqref{Mac} that
\begin{equation}\label{counter}
F_{n}(\zeta)\geq 0,
\end{equation}
for $\zeta\in(-\lambda_{1},0].$
Now it is left to prove the inequality \eqref{fff4} for the case $\zeta>0.$
By using bilinear expansion of the iterated kernels, namely,
$$K_{p}(|x-y|)=\int_{\Omega}...
\int_{\Omega}K(|x-y_{1}|)K(|y_{1}-y_{2}|)...K(|y_{p-1}-y|)dy_{1}...dy_{p-1}$$
and
$$
K_{\zeta,1}(|x-y|):=K_{\zeta}(|x-y|),
$$

$$K_{\zeta,n}(|x-y|)=\int_{\Omega}...
\int_{\Omega}K_{\zeta}(|x-y_{1}|)
K_{\zeta}(|y_{1}-y_{2}|)...
K_{\zeta}(|y_{n-1}-y|)dy_{1}...dy_{n-1},\,n=2,3,\ldots,$$
we obtain
\begin{equation}
\sum_{j=1}^{\infty}\frac{1}{\mu_{j}^{p-1}(\Omega)(\mu_{j}(\Omega)+\zeta)^{n}}
=
\int_{\Omega}\int_{\Omega}K_{p-1}(|x-y|)
K_{\zeta,n}(|x-y|)dxdy,
\end{equation}
for $p=p_{0},p_{0}+1,\ldots,\,n=1,2,\ldots,$ where $p_{0}$ is the smallest integer $\geq q.$
It follows from the Brascamp-Lieb-Luttinger inequality that
\begin{multline}
\int_{\Omega}\int_{\Omega}K_{p-1}(|x-y|)
K_{\zeta,n}(|x-y|)dxdy\leq \\
\int_{B}\int_{B}K_{p-1}(|x-y|)
K_{\zeta,n}(|x-y|)dxdy,
\end{multline}
which proves
\begin{equation}
\sum_{j=1}^{\infty}\frac{1}{|\mu_{j}(\Omega)|^{p-1}(|\mu_{j}(\Omega)|
+\zeta)^{n}}
\leq\sum_{j=1}^{\infty}\frac{1}{|\mu_{j}(B)|^{p-1}(|\mu_{j}(B)|+\zeta)^{n}}
\end{equation}
for any bounded open domain $\Omega\subset \mathbb R^{d}$ with $|\Omega|=|B|$ and an arbitrary $\zeta>0.$
Here we have used that by the assumption $K_{\zeta}$ is a symmetric-decreasing function.
\end{proof}

\begin{lem} \label{pos:f}
The inverse bilateral Laplace transform preserves
the inequality \eqref{n8}.
\end{lem}

\begin{proof} [Proof of Lemma \ref{pos:f}]
One has
\begin{equation}\label{f1}
\mathcal{B}\{f(t)\}(s)=\mathcal{L}\{f(t)\} (s)+\mathcal{L}\{f(-t)\}(-s),
\end{equation}
where $\mathcal{B}$ is the bilateral Laplace transform and
$\mathcal{L}$ is the Laplace transform.
Let us rewrite \eqref{f1} in the form
\begin{equation}\label{f2}
F(s)=F_{1}(s)+E_{1}(-s)
\end{equation}
with
$$F(s)=\mathcal{B}\{f(t)\}(s),\,F_{1}(s)=\mathcal{L}\{f(t)\}(s),\,
E_{1}(-s)=\mathcal{L}\{f(-t)\}(-s).$$
In our case we have
$$F(\zeta)=
\sum_{j=1}^{\infty}\frac{2|\mu_{j}(B)|}{|\mu_{j}(B)|^{p-1}
(|\mu_{j}(B)|^{2}-\zeta^{2})}-\frac{2|\mu_{j}(\Omega)|}{
|\mu_{j}(\Omega)|^{p-1}
(|\mu_{j}(\Omega)|^{2}-\zeta^{2})}
,\,\,|\zeta|<|\mu_{1}(B)|,$$

$$F_{1}(\zeta)=
\sum_{j=1}^{\infty}\frac{1}{|\mu_{j}(B)|^{p-1}
(|\mu_{j}(B)|+\zeta)}-\frac{1}{|\mu_{j}(\Omega)|^{p-1}
(|\mu_{j}(\Omega)|+\zeta)}
,\,\,-|\mu_{1}(B)|<\zeta,$$

$$E_{1}(-\zeta)=
\sum_{j=1}^{\infty}\frac{1}{|\mu_{j}(B)|^{p-1}
(|\mu_{j}(B)|-\zeta)}-\frac{1}{|\mu_{j}(\Omega)|^{p-1}
(|\mu_{j}(\Omega)|-\zeta)}
,\,\,\zeta<|\mu_{1}(B)|.$$
To show the positivity of $f(t)$ it is sufficient to check the conditions \eqref{EQ:Lappos} for $F_{1}$.
From Lemma \ref{+s} we have
\begin{equation}
0\leq\sum_{j=1}^{\infty}\frac{1}{|\mu_{j}(B)|^{p-1}(|\mu_{j}(B)|+\zeta)^{n}}-
\frac{1}{|\mu_{j}(\Omega)|^{p-1}(|\mu_{j}(\Omega)|+\zeta)^{n}},
\end{equation}
for any $n\in \mathbb N$ and $-|\mu_{1}(B)|<\zeta$.
It follows that
$$0\leq(-1)^{(k)}F_{1}^{(k)}(\zeta),\quad k=0,1,2,...,$$
for all $\zeta>-|\mu_{1}(B)|$, which proves the positivity of $f$
(by Theorem \ref{bryan}), that is, $f(t)\geq 0$ for all $t>0.$
Since in our case we have $f(|t|)$ this shows
$$f(t)\geq 0\; \textrm{for all}\;t\in \mathbb R.$$
This confirms that the inverse bilateral Laplace transform preserves
the inequality \eqref{n8}.
\end{proof}

\begin{proof} [Proof of Theorem \ref{THM:main}]
The proof of Theorem \ref{THM:main} now follows directly from Lemma \ref{lem:4}.
Indeed, applying the Mellin transform
$$
\frac{1}{|\mu_{i}|^{l}}=
\frac{1}{\Gamma(l)}\int_{0}^{\infty} \exp(-|\mu_{i}| t)t^{l-1}dt, \quad
\textrm{ for any real } \;  l>1,
$$
to the inequality \eqref{n9} leads to
\begin{equation}\label{n10}
\sum_{i=1}^{\infty}\frac{1}{|\mu_{i}(\Omega)|^{p-1+l}}\leq
\sum_{i=1}^{\infty}\frac{1}{|\mu_{i}(B)|^{p-1+l}}
\end{equation}
for any real $l>1$ and any integer $p\geq q$. As before let $p_{0}$
be the smallest integer $\geq q$. Then since $l$ is arbitrary real number $>1$, from \eqref{n10} we obtain
\begin{equation}
\sum_{i=1}^{\infty}\frac{1}{|\mu_{i}(\Omega)|^{p}}\leq
\sum_{i=1}^{\infty}\frac{1}{|\mu_{i}(B)|^{p}}
\end{equation}
for any real $p> p_{0}$. In addition, by Lemma \ref{lem:2} we confirm that
the inequality is also true when $p=p_{0}$.
This completes the proof of Theorem \ref{THM:main}.
\end{proof}

\begin{rem}
Our main techniques are the Brascamp-Lieb-Luttinger inequality for multiple integrals,
and positive integral representations of the trace
of the $p$-th power of the operator, that is, $p_{0}$ being an integer is important. Therefore,
the techniques do not allow us to prove Theorem \ref{THM:main} for $q\leq p<p_{0}$.
In view of the Dirichlet Laplacian case,
it seems reasonable to conjecture that
the Schatten $p$-norm is still
maximized on the ball also for $q\leq p<p_{0}$. In addition, we also conjecture that Theorem \ref{THM:main} is valid without the assumption
on the existence of $\zeta$-kernel.
\end{rem}

\section{The case of polygons}
\label{SEC:6}
We can ask the same question of maximizing the Schatten $p$-norms in the class of polygons with a
given number $n$ of sides. We denote by $\mathcal{P}_{n}$ the class of plane polygons with $n$ edges.
We would like to identify the maximizer for Schatten $p$-norms of the convolution type positive compact
operator $\mathcal{K}_{\Omega}$ in $\mathcal{P}_{n}$.
According to Section \ref{SEC:result}, it is natural to conjecture that it is the $n$-regular polygon.
Currently, we can prove it only for $n=3$:

\begin{thm}\label{THM:tri}
Let $\Delta$ be an equilateral triangle centered at the origin and
assume that $\mathcal{K}_{\Delta}\in\mathcal{S}^{q}(L^{2}(\Delta))$
for some $q> 1$. Let $\Omega$ be any triangle with $|\Omega|=|\Delta|.$ Assume that $\mathcal{K}_{\Omega}$ and $\mathcal{K}_{\Delta}$ are positive operators admitting $\zeta$-kernels.
Then
\begin{equation}
\|\mathcal{K}_{\Omega}\|_{p}\leq  \|\mathcal{K}_{\Delta}\|_{p}
\label{3-2}
\end{equation}
for any $p_{0}\leq p< \infty$, where $p_{0}$ is the smallest integer $\geq q.$
\end{thm}

Let $u$ be a nonnegative, measurable function on $\mathbb{R}^{d}$, and let $V$ be a $d-1$ dimensional plane through the origin of $\mathbb{R}^{d}$. Choose an orthogonal coordinate system in $\mathbb{R}^{d}$ such that the $x^{1}$-axis is perpendicular to $V\ni z=(x^{2},\ldots,x^{d})$.
\begin{defn}[\cite{BLL}]
A nonnegative, measurable function $u^{\star}(x|V)$ on $\mathbb{R}^{d}$ is called a Steiner symmetrization with respect to $V$ of the function $u(x)$, if $u^{\star}(x^{1},x^{2},\ldots,x^{d})$ is a symmetric decreasing rearrangement with respect to $x^{1}$ of $u(x^{1},x^{2},\ldots,x^{d})$ for each fixed $x^{2},\ldots,x^{d}$.
\end{defn}

The Steiner symmetrization (with respect to the $x^{1}$-axis) $\Omega^{\star}$
of a measurable set $\Omega$ is defined in the following way:
if we write $x=(x^{1},z)$ with $z\in \mathbb{R}^{d-1}$, and let
$\Omega_{z}=\{x^{1}: (x^{1},z)\in\Omega\}$, then
$$
\Omega^{\star}:=\{(x^{1},z)\in \mathbb{R}\times\mathbb{R}^{d-1}: x^{1}\in \Omega^{*}_{z}\},
$$
where $\Omega^{*}_{z}$ is a symmetric rearrangement of $\Omega_{z}$
(see the proof of Theorem \ref{THM:1}).

Since the Steiner symmetrization has the same property as the symmetric-decreasing rearrangement (see \cite[Lemma 3.2]{BLL}), that is,
for a non-increasing positive function $K(|x-y|)$ and a positive function $u$ we have

\begin{multline}
\int_{\Omega}\int_{\Omega}\int_{\Omega}
u(y)K(|y-z|)K(|z-x|)u(x)dzdydx\leq
\\
\int_{\Omega^{\star}}\int_{\Omega^{\star}}
\int_{\Omega^{\star}}u^{\star}(y)K(|y-z|)K(|z-x|)u^{\star}(x)dzdydx,\label{nn1}
\end{multline}
this will be useful for our analysis.

Indeed, first we obtain
the following analogue of the P{\'o}lya theorem (\cite{Po}) for the integral operator
$\mathcal{K}_\Omega$ for triangles $\Omega$. See also the book by
P{\'o}lya and Szeg{\"o} \cite{Po1}.

\begin{thm} \label{THM:2}
The equilateral triangle  $\Delta$ centred at the
origin is a minimizer of the first characteristic number of the convolution type compact operator $\mathcal{K}_\Omega$ among all triangles of a given area, i.e.
$$0<\mu_{1}(\Delta)\leq \mu_{1}(\Omega)$$
for any triangle $\Omega\subset \mathbb R^{2}$ with $|\Omega|=|\Delta|.$
\end{thm}

\begin{rem}\label{REM:op2}
In other words Theorem \ref{THM:2} says that the operator norm of
$\mathcal{K}_\Omega$ is maximized in an equilateral triangle among all triangles of a given area.
\end{rem}

\begin{proof} [Proof of Theorem \ref{THM:2}]
By Remark \ref{REM:op} the first characteristic number $\mu_{1}$ of the operator
$\mathcal{K}$ is positive and simple; the corresponding
eigenfunction $u_{1}$ can be chosen positive in $\Omega$.
Using the fact that applying a sequence of Steiner symmetrizations
with respect to the mediator of each side, a given triangle transforms into an equilateral one
(see Figure 3.2. in \cite{He}), and the fact that $K(|x-y|)$ is a non-increasing function,
from \eqref{nn1} we obtain
\begin{multline}
\int_{\Omega}\int_{\Omega}\int_{\Omega}
u_{1}(y)K(|y-z|)K(|z-x|)u_{1}(x)dzdydx\leq\\
\int_{\Delta}\int_{\Delta}
\int_{\Delta}u_{1}^{\star}(y)K(|y-z|)K(|z-x|)u_{1}^{\star}(x)dzdydx.\label{nn2}
\end{multline}

Therefore, from \eqref{nn2} and the variational principle for the positive operator $\mathcal{K}^{2}_\Delta$, we get
$$
\mu_{1}^{2}(\Omega)=\frac{\int_{\Omega}|u_{1}(x)|^{2}dx}{\int_{\Omega}\int_{\Omega}\int_{\Omega}
u_{1}(y)K(|y-z|)K(|z-x|)u_{1}(x)dzdydx}\geq$$
$$\frac{\int_{\Delta}|u^{\star}_{1}(x)|^{2}dx}{\int_{\Delta}\int_{\Delta}\int_{\Delta}
u^{\star}_{1}(y)K(|y-z|)K(|z-x|)u^{\star}_{1}(x)dzdydx}\geq
$$
$$
\inf_{v\in
L^{2}(\Delta),v\not\equiv 0}\frac{\int_{\Delta}|v(x)|^{2}dx}{\int_{\Delta}\int_{\Delta}\int_{\Delta}
v(y)K(|y-z|)K(|z-x|)v(x)dzdydx}=\mu_{1}^{2}(\Delta).
$$
Since both $\mu_{1}(\Omega)$ and $\mu_{1}(\Delta)$ are
positive by Remark \ref{REM:op}, we arrive at
$$\mu_{1}(\Omega)\geq \mu_{1}(\Delta)>0$$
for any triangle $\Omega\subset \mathbb R^{2}$ with $|\Omega|=|\Delta|.$
Here we have used the fact that the Steiner symmetrization preserves the $L^{2}$-norm.
The proof is complete.
\end{proof}

\begin{proof} [Proof of Theorem \ref{THM:tri}]
The proof of Theorem \ref{THM:tri} relies on the same technique as Theorem \ref{THM:main} with
the difference that now we use the Steiner symmetrization.
Since the Steiner symmetrization has the same property \eqref{eq16} (see
\cite[Lemma 3.2]{BLL}) as the symmetric-decreasing rearrangement, it is
clear that any Steiner symmetrization increases (or at least does not decrease) the Schatten $p$-norms (cf. Lemma \ref{lem:2}) for integer $p\geq q$.
We only need to recall the fact that by a sequence of Steiner symmetrizations
with respect to the mediator of each side, a given triangle converges to an equilateral one.
The rest of the proof is exactly the same as the proof of Theorem \ref{THM:main}.
\end{proof}

\end{document}